\documentclass[12pt]{article}

\usepackage{amsmath,amssymb,amsbsy,amsfonts,amsthm,latexsym,
            amsopn,amstext,amsxtra,amscd,stmaryrd,fullpage}

\usepackage[mathscr]{eucal}
\usepackage{color}
\begin{document}

\newtheorem{theorem}{Theorem}
\newtheorem{lemma}[theorem]{Lemma}
\newtheorem{corollary}[theorem]{Corollary}
\newtheorem{proposition}[theorem]{Proposition}

\theoremstyle{definition}
\newtheorem*{definition}{Definition}
\newtheorem*{remark}{Remark}
\newtheorem*{example}{Example}


\def\cA{\mathcal A}
\def\cB{\mathcal B}
\def\cC{\mathcal C}
\def\cD{\mathcal D}
\def\cE{\mathcal E}
\def\cF{\mathcal F}
\def\cG{\mathcal G}
\def\cH{\mathcal H}
\def\cI{\mathcal I}
\def\cJ{\mathcal J}
\def\cK{\mathcal K}
\def\cL{\mathcal L}
\def\cM{\mathcal M}
\def\cN{\mathcal N}
\def\cO{\mathcal O}
\def\cP{\mathcal P}
\def\cQ{\mathcal Q}
\def\cR{\mathcal R}
\def\cS{\mathcal S}
\def\cU{\mathcal U}
\def\cT{\mathcal T}
\def\cV{\mathcal V}
\def\cW{\mathcal W}
\def\cX{\mathcal X}
\def\cY{\mathcal Y}
\def\cZ{\mathcal Z}


\def\sA{\mathscr A}
\def\sB{\mathscr B}
\def\sC{\mathscr C}
\def\sD{\mathscr D}
\def\sE{\mathscr E}
\def\sF{\mathscr F}
\def\sG{\mathscr G}
\def\sH{\mathscr H}
\def\sI{\mathscr I}
\def\sJ{\mathscr J}
\def\sK{\mathscr K}
\def\sL{\mathscr L}
\def\sM{\mathscr M}
\def\sN{\mathscr N}
\def\sO{\mathscr O}
\def\sP{\mathscr P}
\def\sQ{\mathscr Q}
\def\sR{\mathscr R}
\def\sS{\mathscr S}
\def\sU{\mathscr U}
\def\sT{\mathscr T}
\def\sV{\mathscr V}
\def\sW{\mathscr W}
\def\sX{\mathscr X}
\def\sY{\mathscr Y}
\def\sZ{\mathscr Z}


\def\fA{\mathfrak A}
\def\fB{\mathfrak B}
\def\fC{\mathfrak C}
\def\fD{\mathfrak D}
\def\fE{\mathfrak E}
\def\fF{\mathfrak F}
\def\fG{\mathfrak G}
\def\fH{\mathfrak H}
\def\fI{\mathfrak I}
\def\fJ{\mathfrak J}
\def\fK{\mathfrak K}
\def\fL{\mathfrak L}
\def\fM{\mathfrak M}
\def\fN{\mathfrak N}
\def\fO{\mathfrak O}
\def\fP{\mathfrak P}
\def\fQ{\mathfrak Q}
\def\fR{\mathfrak R}
\def\fS{\mathfrak S}
\def\fU{\mathfrak U}
\def\fT{\mathfrak T}
\def\fV{\mathfrak V}
\def\fW{\mathfrak W}
\def\fX{\mathfrak X}
\def\fY{\mathfrak Y}
\def\fZ{\mathfrak Z}


\def\C{{\mathbb C}}
\def\F{{\mathbb F}}
\def\K{{\mathbb K}}
\def\L{{\mathbb L}}
\def\N{{\mathbb N}}
\def\Q{{\mathbb Q}}
\def\R{{\mathbb R}}
\def\Z{{\mathbb Z}}
\def\E{{\mathbb E}}
\def\T{{\mathbb T}}
\def\P{{\mathbb P}}
\def\D{{\mathbb D}}


\def\eps{\varepsilon}
\def\mand{\qquad\mbox{and}\qquad}
\def\\{\cr}
\def\({\left(}
\def\){\right)}
\def\[{\left[}
\def\]{\right]}
\def\<{\langle}
\def\>{\rangle}
\def\fl#1{\left\lfloor#1\right\rfloor}
\def\rf#1{\left\lceil#1\right\rceil}
\def\le{\leqslant}
\def\ge{\geqslant}
\def\ds{\displaystyle}

\def\xxx{\vskip5pt\hrule\vskip5pt}
\def\yyy{\vskip5pt\hrule\vskip2pt\hrule\vskip5pt}
\def\imhere{ \xxx\centerline{\sc I'm here}\xxx }

\newcommand{\comm}[1]{\marginpar{
\vskip-\baselineskip \raggedright\footnotesize
\itshape\hrule\smallskip#1\par\smallskip\hrule}}


\def\e{\mathbf{e}}
\def\sPrc{{\displaystyle \sP_r^{(c)}}}

\title{\bf  Zeros of random linear combinations of entire functions with complex Gaussian coefficients}

\author{
{\sc Aaron M.~Yeager} \\
{Department of Mathematics, Oklahoma State University} \\
{Stillwater, OK 74078 USA} \\
{\tt aaron.yeager@okstate.edu}}

\maketitle

\begin{abstract}
We study zero distribution of random linear combinations of the form
$$P_n(z)=\sum_{j=0}^n\eta_jf_j(z),$$
in any Jordan region $\Omega \subset \mathbb C$. The basis functions $f_j$ are entire functions that are real-valued on the real line, and $\eta_0,\dots,\eta_n$ are complex-valued iid Gaussian random variables.  We derive an explicit intensity function for the number of zeros of $P_n$ in $\Omega$ for each fixed $n$.  Our main application is to polynomials orthogonal on the real line. Using the Christoffel-Darboux formula the intensity function takes a very simple shape.  Moreover, we give the limiting value of the intensity function when the orthogonal polynomials are associated to Szeg\H{o} weights.
\end{abstract}


\section{Introduction}

The systematic study of zeros of polynomials $P_n(z)=\sum_{j=0}^n\eta_jz^j$ with random coefficients, called \emph{random algebraic polynomials}, dates back to the 1932 paper due to Bloch and P\'{o}lya \cite{BP}.  They showed that when $\{\eta_j\}$ are iid random variables that take values from the set $\{-1,0,1\}$ with equal probabilities, the expected number of real zeros is $O(\sqrt{n})$.  Early advancements in the subject were then made by Littlewood and Offord \cite{LO}, Kac \cite{K1}, \cite{K2}, Rice \cite{R}, Erd\H{o}s and Offord \cite{EO}, and many others.  A nice history of the early progress in this topic is given by Bharucha-Reid and Sambandham \cite{BRS} and by Farahmand \cite{FB}.

Kac \cite{K1}, \cite{K2}  produced a formula that gives the expected number of real zeros of $P_n(z)$ when each $\eta_j$ are independent real-valued Gaussian coefficients.  From that formula, he was able to show that the expected number of real roots of the random algebraic polynomial is asymptotic to $2\pi^{-1}\log n$ as $n\rightarrow \infty$.  The error term in Kac's asymptotic was further sharpened by Hammersley \cite{HM}, Wang \cite{WG}, Edelman and Kostlan \cite{EK}, and Wilkins \cite{WL}.

Shepp and Vanderbei \cite{SV} extended Kac's formula in 1995  by giving a formula in the complex plane for the expected number of zeros of a random algebraic polynomial when the random variables are real-valued independent standard Gaussian. They were also able to obtain a limit of the intensity  function (density function of zeros) as $n\rightarrow \infty$.  Generalizations were made of their result by Ibragimov and Zeitouni \cite{IZ} for random variables from the domain of attraction of a stable law.

In 1996,  Farahmand \cite{F} produced a formula for the intensity function for random algebraic polynomials when the random coefficients are complex-valued iid Gaussian random variables. As an application, Farahmand  considered the spanning functions of the random polynomial to be  monomials and also the cosine functions.  Using a theorem of Adler (Theorem 5.1.1, p. 95 of \cite{A}) in 1998,  Farahmand and Jahangiri \cite{KJ} extended the previous result by Farahmand by putting weights $\{g_j\}_{j=0}^n\subset \R $ on the random algebraic polynomials.
Then in 2001 Farahmand and Grigorash \cite{FG}, also using Adler's theorem, gave a formula for the intensity function and its limiting value when the random polynomials are spanned by the cosine functions.

Some authors have studied the intensity function for what are known as Gaussian Analytic Functions (GAF), $P(z)=\sum_{j=0}^{\infty}\eta_j f_j(z)$ where the $f_j$'s are square summable analytic functions on a domain, and the $\eta_j$'s are iid Gaussian random variables, in terms of the distributional Laplacian.  In 2000, a formula for the intensity function of a GAF was given by Hough, Krishnapur, Peres, and Vir$\acute{\text{a}}$g in \cite{ZGAF} (Section 2.4.2, pp. 25-26) when the $\eta_j$'s are complex-valued iid Gaussian random variables.  Also during that year, Feldheim \cite{FL} gave the same formula (Theorem 2, p. 6) and a formula (Theorem 3, p. 7) for the intensity function of a GAF when the $f_j$'s are real-valued on the real line and the $\eta_j$'s are real-valued independent standard Gaussian random variables.

Recently in 2015 Vanderbei \cite{CZRS} produced an explicit formula of an intensity function for  finite sums of real-valued iid standard Gaussian random variables with the spanning functions taken to be entire functions that are real-valued on the real line.  Vanderbei also gave the limiting intensity function when the spanning functions are Weyl polynomials, Taylor polynomials, and the truncated Fourier series.

In this paper, following a similar method of proof as given by Vanderbei in \cite{CZRS}, we derive an explicit formula for a class of finite linear combinations of entire functions with complex-valued iid Gaussian coefficients.  To specify our results, let $\{f_j(z)\}_{j=0}^n$ be a sequence of entire functions in the complex plane that are real-valued on the real line.  We will be studying the expectation of the number zeros of random polynomials of the form
\begin{equation}\label{P}
P_n(z)=\sum_{j=0}^n \eta_j f_j(z), \ \ \ \ z\in\C,
\end{equation}
where $n$ is a fixed integer, and $\eta_j=\alpha_j + i \beta_j$, $j=0,1, \dots, n$, with $\{\alpha_j\}_{j=0}^n$ and $\{\beta_j\}_{j=0}^n$ being sequences of iid standard normal random variables. The formula we derive for the intensity function is expressed in terms of the kernels
\begin{equation}\label{K01}
K_{n}(z,w)=\sum_{j=0}^{n}f_j(z)\overline{f_j(w)},\ \ \ \ \ \ \ K_{n}^{(0,1)}(z,w)=\sum_{j=0}^{n} f_j(z)\overline{f_j^{\prime}(w)},
\end{equation}
and
\begin{equation}\label{K2}
K_{n}^{(1,1)}(z,w)=\sum_{j=0}^{n}f_j^{\prime}(z)\overline{f_j^{\prime}(w)}.
\end{equation}
We note that since our functions $f_j(z)$ are entire functions that are real-valued on the real line, by the Schwarz Reflection Principle we have $\overline{f_j(z)}=f_j(\bar{z})$ for all $j=0,1,\dots,n$, and all $z\in \C$.

Let $N_n(\Omega)$ denote the (random) number of zeros of $P_n(z)$ as defined by \eqref{P} in a Jordan region $\Omega$ of the complex plane.  Our formula for the intensity function is the following:

\begin{theorem}\label{thm2.1}
Let $P_n(z)$ be the random sum \eqref{P} spanned by entire functions that are real-valued on the real line with complex-valued iid Gaussian coefficients. For each Jordan region $\Omega \subset \{z\in \C : K_{n}(z,z)\neq 0\}$, we have that the intensity function $h_n$ satisfies
$$\E[N_n(\Omega)]=\int_{\Omega}h_n(x,y)\ dx \ dy ,$$
with
$$h_n(x,y)=h_n(z)=\frac{K_{n}^{(1,1)}(z,z)K_{n}(z,z)-\left|K_{n}^{(0,1)}(z,z)\right|^2}{\pi \left(K_{n}(z,z)\right)^2},$$
where the kernels $K_n(z,z)$, $K_n^{(0,1)}(z,z)$, and $K_n^{(1,1)}(z,z)$, are defined in \eqref{K01} and \eqref{K2}.
\end{theorem}

We note that since all the functions that make up $h_n$ are real valued, the function $h_n$ is real valued.  Furthermore, the function $h_n$ is in fact nonnegative.

The formula in Theorem \ref{thm2.1} and the other results of this paper were posted as a preprint to arXiv on May 22, 2016 \cite{AY}, and were presented at the $15^{\text{th}}$ International Conference in Approximation Theory in San Antonio, TX, on May 25, 2016.  On May $25^{\text{th}}$ at the conference another participant, Andrew Ledoan \cite{AL}, also presented the result of Theorem \ref{thm2.1} in a slightly different form.  We also note that the result in Theorem \ref{thm2.1} can be shown to be the same formula given by Hough, Krishnapur, Peres, and Vir$\acute{\text{a}}$g in \cite{ZGAF} and also given by Feldheim \cite{FL}. However, our method of proof for deriving the intensity function is different than in \cite{ZGAF} and \cite{FL}.

Studying the case when the spanning functions $f_j(z)$ of \eqref{P} are polynomials orthogonal on the real line (OPRL) has been considered by many authors.  We say that a collection of polynomials $\{p_j(z)\}_{j\geq 0}$ are \emph{orthogonal on the real line with respect to $\mu$}, with $\text{supp}\ \mu \subseteq \R$, if
$$\int p_n(x)p_m(x)d\mu(x)=\delta_{nm}, \ \ \ \ \text{for all $n,m\in \N \cup\{0\}$}.$$
We note that when polynomials are orthogonal on the real line, they have real coefficients and are real-valued on the real line.

In 1971 Das \cite{D} showed that when the OPRL are Legendre polynomials and the random variables are real-valued iid Gaussian, the average number of zeros of the random orthogonal polynomial in $(-1,1)$ is asymptotic to $n/\sqrt{3}$ when $n$ is large.  Das and Bhatt \cite{DB} extended this result in 1982 to include the class of OPRL to be the classicial orthogonal polynomials, Jacobi, Laguerre, and Hermite, and showed the same asymptotic held true for the zeros of the random orthogonal polynomial in $(-1,1)$. Their results concerning the Hermite and Laguerre had some gaps.  These gaps were fixed in 2015 by Lubinsky, Pritsker, and Xie \cite{LPX} by considering a larger class of OPRL that had only mild assumptions on the measure and weight function.  Using potential theory for their results, they showed that the same asymptotic holds for the larger class of OPRL.   These results were further generalized by Lubinsky, Pritsker, and Xie \cite{LPX2} to allow the OPRL to have support on the whole real line, with the same asymptotic.

Other authors have considered the case when the random polynomial is spanned by orthogonal polynomials that satisfy orthogonality relations on curves or domains.   There has also been work done in the higher dimensional analogs of these settings, see Shiffman and Zelditch \cite{SHZ1}-\cite{SHZ3}, Bloom \cite{BL1} and \cite{BL2}, Bloom and Shiffman \cite{BLSH}, Bloom and Levenberg \cite{BLL}, and Bayraktar \cite{BY}.

We focus our applications on OPRL. Using the Christoffel-Darboux formula we show that the intensity function from Theorem \ref{thm2.1} greatly simplifies when the spanning functions are taken to be OPRL.  To be more explicit, the random orthogonal polynomials we will apply Theorem \ref{thm2.1} to are of the form
\begin{equation}\label{ROS}
P_n(z)=\sum_{j=0}^n \eta_j p_j(z), \ \ \ \ z\in\C,
\end{equation}
where $n$ is a fixed integer, $\{p_j(z)\}$ are OPRL, and $\eta_j=\alpha_j + i \beta_j$, $j=0,1, \dots, n$, with $\{\alpha_j\}_{j=0}^n$ and $\{\beta_j\}_{j=0}^n$ being sequences of iid standard normal random variables.

\begin{theorem}\label{OP}
For the random orthogonal polynomial \eqref{ROS}, the intensity function $h_n^P(z)$ as defined in Theorem \ref{thm2.1} simplifies to
\begin{align*}
h_n^P(z)
&=\frac{1}{4\pi\left(\emph{\text{Im}}(z)\right)^2}-\frac{|p_{n+1}^{\prime}(z)p_n(z)-p_n^{\prime}(z)p_{n+1}(z)|^2}{4\pi \left(\emph{\text{Im}}(p_{n+1}(z)p_n(\bar{z})  \right)^2},\ \ \ \ z\in\C.
\end{align*}
\end{theorem}
Although the intensity $h_n^P$ has the $\text{Im}(z)$ and $\left(\text{Im}(p_{n+1}(z)p_n(\bar{z})  \right)^2$ in the denominator, these singularises cancel out algebraically when $z\in \R$ and $p_{n+1}(z)p_n(\bar{z})\in \R$.  Hence the intensity function $h_n^P$ is well defined on $\R$ also.
We also note that under the assumptions of Theorem \ref{OP}, the intensity function formula of Theorem \ref{thm2.1} holds for all $z\in \C$. This follows since $p_0(z)$ is a non-zero constant, so that $K_n(z,z)=\sum_{j=0}^n|p_j(z)|^2>0$.  Thus our intensity function is well defined and continuous everywhere on $\C$.

We conclude the paper by giving the limiting value of the intensity function for a class of random orthogonal polynomials. We say that $f(\theta)\geq 0$ belongs to the \emph{Szeg\H{o} weight class}, denoted by $G$, if $f(\theta)$ is defined and measurable in $[-\pi,\pi]$, and the integrals
$$\int_{-\pi}^{\pi} f(\theta)\ d\theta, \ \ \ \ \int_{-\pi}^{\pi}|\log f(\theta)|\ d\theta$$
exist with the first integral assumed to be positive.     When $w(x)$ is a weight function supported on $[-1,1]$ with $w(\cos \theta)|\sin \theta|=f(\theta)\in G$, the \emph{orthogonal polynomials associated to $w(x)$} are polynomials $\{p_j(z)\}_{j\geq 0}$, where $z\in \C$, with real coefficients such that
$$\int_{-1}^1 p_n(x)p_m(x)w(x)dx=\delta_{nm}, \ \ \  \text{for all $n,m\in \N \cup\{0\}$}.$$
Taking $w(\cos \theta)|\sin \theta|=f(\theta)\in G$ and using asymptotics by Szeg\H{o} in \cite{SZ} (Theorems 12.1.1 and 12.1.2, p. 297), and then appealing to our Theorem \ref{OP}, we are able to obtain the limiting value of the intensity function for the random orthogonal polynomials associated to $w(x)$.

\begin{theorem}\label{GOP2}
Let $w(x)$ be a weight function on the interval $-1\leq x\leq 1$ such that $w(\cos \theta )|\sin \theta |=f(\theta)$ belongs to the weight class $G$. The intensity function for the random orthogonal polynomial \eqref{ROS} with complex-valued iid Gaussian coefficients, where the $p_j$'s are associated to $w(x)$, satisfies
\begin{equation}\label{IOA}
\lim_{n\rightarrow \infty}h_n^P(z)= \frac{1}{4\pi \left(\emph{\text{Im}}(z)\right)^2}\left(1-\frac{\left(\emph{\text{Im}}(z)\right)^2\left|z+\sqrt{z^2-1}\right|^2}{|z^2-1|\left( \emph{\text{Im}}(z+\sqrt{z^2-1} \right)^2} \right),
\end{equation}
 for all $z\in \C\setminus[-1,1]$.  Furthermore, convergence in \eqref{IOA} holds uniformly on compact subsets of $\C\setminus [-1,1]$.
\end{theorem}

\section{The Intensity Function $h_n$}

As mentioned, to prove Theorem \ref{thm2.1} we follow the method of proof given by Vanderbei in \cite{CZRS}.  Many of the parts of our proof of this theorem are nearly identical to that given by Vanderbei in \cite{CZRS}. The major differences in our proof are that we must consider the set $\{z\in \C: K_{n}(z,z)=0\}$, and since the random variables are complex valued, we will work with their real and imaginary parts, which in turn will make the covariance matrix we compute different than the one in the result by Vanderbei.  Taking these adjustments into account, for convenience of the reader we present the proofs of the needed  lemmas and the proof of the main theorem which gives the formula for the intensity function $h_n$.

Our first needed lemma is the following:
\begin{lemma}\label{lem3.1}
For each Jordan region $\Omega \in \C$ whose boundary intersects the set $\{z\in \C: K_{n}(z,z)=0\}$ at most only finitely many times, we have
\begin{equation}\label{F}
\E[N_n(\Omega)]=\frac{1}{2\pi i}\int_{\partial \Omega}F(z)\ dz,
\end{equation}
where
\begin{equation*}
F(z)=\E\left[ \frac{ P_n^{\prime}(z)}{ P_n(z)}\right]=\frac{K_{n}^{(0,1)}(\bar{z},\bar{z})}{K_{n}(z,z)}.
\end{equation*}
\end{lemma}
\begin{proof}
Using the argument principle in \cite{Ul} on page 79, we have an explicit formula for the random variable $N_n(\Omega)$ given by
$$N_n(\Omega)=\frac{1}{2\pi i}\int_{\partial \Omega} \frac{P_n^{\prime}(z)}{P_n(z)}dz. $$
If we now take the expectation and then interchange the expectation and the contour integral via the complex form of the Fubini-Tonelli Theorem (Theorem 8.8, p. 164 of \cite{RU}) (we give a sketch of the justification for this when the spanning functions $f_j$ are polynomials $p_j$, with $\deg p_j=j$ for all $j$, below) we obtain
$$\E[N_n(\Omega)]=\frac{1}{2\pi i}\int_{\partial \Omega} \E\left[ \frac{P_n^{\prime}(z)}{P_n(z)}\right]dz. $$
In the next lemma we show that for $z\not\in \{z:K_{n}(z,z)=0\}$, we have
$$\E\left[ \frac{ P_n^{\prime}(z)}{ P_n(z)}\right]=\frac{K_{n}^{(0,1)}(\bar{z},\bar{z})}{K_{n}(z,z)}.$$
Taking this into account and since we have supposed that $\partial\Omega$ intersects the set $\{z\in \C: K_{n}(z,z)=0\}$ at most only finitely many times, once the justification of the use of the Fubini-Tonelli Theorem is done, the proof will be complete.

We now provide a sketch of the proof for the justification of the interchange the expectation and the contour integral which has been deemed as ``tedious but doable" by many authors for the case that our applications are in.  That is, the case when the spanning functions $f_j$ are polynomials $p_j$ with $\deg p_j=j$ for all $j$.  Since $P_n$ and $P_n^{\prime}$ both depend on the random variables $\eta_0, \eta_1,\dots,\eta_n$, we write $P_n(z,\vec{\eta})$ and $P_n^{\prime}(z,\vec{\eta})$ where $\vec{\eta}:=(\eta_0,\dots,\eta_n)$.  We denote the joint density function for the random variables $\{\eta_j\}_{j=0}^n$ by $f(\vec{\eta})$. To use the Fubini-Tonelli Theorem we must show $\left|\frac{P_n^{\prime}(z,\vec{\eta})}{P_n(z,\vec{\eta})}\right|$ is measurable and that
\begin{equation}\label{bdd}
\int_{\R^{2(n+1)}}\int_{\partial \Omega}\left|\frac{P_n^{\prime}(z,\vec{\eta})}{P_n(z,\vec{\eta})}\right| \ |dz|\ f(\vec{\eta})\  dV <\infty.
\end{equation}

Observe that $\E[N_n(\Omega)]$ can be recovered from the values of $\E[N_n(D)]$, where $D:=\{z\in \C : |z|<R\}$.  Thus we will focus on proving \eqref{bdd} for the disk $D$.

We will first prove \eqref{bdd} for $P_n(z)=\sum_{j=0}^n\eta_jz^j$. Let $\{\zeta_k\}_{k=1}^n=\{\zeta_k(\vec{\eta})\}_{k=1}^n$ be the zeros of $P_n(z)$.  Then
$$\left|\frac{P_n^{\prime}(z)}{P_n(z)}\right|=\left|\sum_{k=1}^n\frac{1}{z-\zeta_k}   \right|\leq\sum_{k=1}^n\left|\frac{1}{z-\zeta_k}\right|.$$

Notice that if we can show that all the quotients $\left|\frac{1}{z-\zeta_k}\right|$, $k=1,\dots,n$, satisfy
\begin{equation}\label{bdd3}
\int_{\R^{2(n+1)}}\int_{\partial D}\left|\frac{1}{z-\zeta_k(\vec{\eta})}\right| \ |dz|\ f(\vec{\eta})\  dV <\infty,
\end{equation}
appealing to the linearity of the integration and the triangle inequality we will have \eqref{bdd}.

From Vieta's formulas we have an almost everywhere differentiable change of variables $\phi$ from the set of
roots to the set of coefficients.
For notational sake, we write $dV_{k}$ to stand for the volume measure in $\C^{k}$.  The integration will be done as follows:
\begin{align*}
&\int_{\C^{(n+1)}}\int_{\partial D}\left|\frac{1}{z-\zeta_k(\eta_n,\dots,\eta_n)}\right|\ |dz|\ f(\eta_0,\dots,\eta_n)\  dV_{n+1} \\
&=\int_{\C} \int_{\C^{n}}\int_{\partial D}\left|\frac{1}{z-\zeta_k(\eta_n,\dots,\eta_{n})}\right|\ |dz|\ f(\eta_0,\dots,\eta_{n-1})\  dV_n\ f(\eta_n) \ d\eta_n.
\end{align*}

Fixing the outer variable $\eta_n$, we have $\phi:\C^n\rightarrow \C^n$ with Jacobian determinant
$$|\eta_n|^{2n}\prod_{1\leq i<j\leq n}\left|\zeta_i-\zeta_j\right|^2:=J(\phi).$$

Let us first consider the case when an arbitrary zero $\zeta_k(\vec{\eta})$ of $P_n(z)$ lies on the contour $\partial D$.  Define
$A:=\{\text{coefficients of $P_n(z)$ :  $\zeta_k(\vec{\eta})\in \partial D$}\}.$  Since
$\phi(\{\text{roots of $P_n(z)$}:\zeta_k(\vec{\eta})\in \partial D\})=A$,
using the change of variables formula it follows that
\begin{align*}
V_n(A)&
=\int_{\{\text{roots of $P_n(z)$}:\zeta_k(\vec{\eta})\in \partial D\}} \phi(\zeta_k)J(\phi)\  dV_n
\leq \int_{\prod_{k=1}^n\partial D(0,R)} \phi(\zeta_k)J(\phi)\  dV_n
=0,
\end{align*}
where the last equality follows since $V_n\left(\prod_{k=1}^n\partial D(0,R)\right)=0$. Hence we have $V_n(A)=0$, and consequently $V_{n+1}(A)=0$.  Therefore the set $A$ is negligible for \eqref{bdd3} to hold, so that $|P_n^{\prime}(z)/P_n(z)|$ is measurable.

Let us now go to the case when the zeros of $P_n(z)$ do not lie on the contour $\partial D$ that is being integrated over.
When $|z-\zeta_k(\vec{\eta})|\geq 1$ for $z\in \partial D$, trivially we have $\left|\frac{1}{z-\zeta_k(\vec{\eta})}\right|\leq 1$, so that
\begin{align*}
\int_{\R^{2(n+1)}}\int_{\partial D}\left|\frac{1}{z-\zeta_k(\vec{\eta})}\right| \ |dz|\ f(\vec{\eta})\  dV
&\leq 2\pi R<\infty.
\end{align*}

When $0<|z-\zeta_k(\vec{\eta})|< 1$, carefully estimating the integrand it follows that
\begin{align*}
\int_{\partial D}\left|\frac{1}{z-\zeta_k(\vec{\eta})}\right| \ |dz|&
 \leq C\log \left(\frac{1}{|R-|\zeta_k||}\right),
\end{align*}
where $C>0$ is a constant that depends only on $R$.

Taking this into account and using the change of variables with the map $\phi$ yields
\begin{align*}
&C\int_{\R^{2(n+1)}} \log \left(\frac{1}{|R-|\zeta_k||}\right)\ f(\vec{\eta})\  dV \\
&=\frac{C}{\pi^{n+1}}\int_{\R^2}\int_{\R^{2n}} \log \left(\frac{1}{|R-|\zeta_k||}\right)\  f(\phi(\eta_0,\dots,\eta_{n-1}))|\eta_n|^{2n}\prod_{1\leq i<j\leq n}\left|\zeta_i-\zeta_j\right|^2\  dV_n \ f(\eta_n)\ dV_1.
\end{align*}
Using polar integration and further estimating, the above integral is less than or equal to the sum of
$$\frac{2K_1}{\pi^{n}}\int_{\R^{2(n+1)}} |\eta_n|^{2n}\prod_{1\leq i<j\leq n\atop i,j\neq k}\left|\zeta_i-\zeta_j\right|^2\prod_{i=1 \atop i\neq k}^n(R+1+|\zeta_i|)^2  \max_{\{\zeta_k\in \overline{D}\}}g(\zeta)  f(\eta_n)\ dV$$
and
$$\frac{4K_2}{\pi^{n}}\int_{\R^{2(n+1)}} |\eta_n|^{2n}\prod_{1\leq i<j\leq n\atop i,j\neq k}\left|\zeta_i-\zeta_j\right|^2\prod_{i=1 \atop i\neq k}^n(R+1+|\zeta_i|)^2  \max_{\{\zeta_k\in \overline{D(0,R+1)}\setminus D \}}g(\zeta)  f(\eta_n)\ dV,$$
where $g(\zeta)=f(\phi(\eta_0,\dots,\eta_{n-1}))$, and $K_1$ and $K_2$ are constants that depend only on $R$.  Notice that both
$$\max_{\{\zeta_k\in \overline{D}\}}g(\zeta) \ \ \ \text{and} \ \ \ \max_{\{\zeta_k\in \overline{D(0,R+1)}\setminus D \}}g(\zeta)$$
decay at infinity exponentially with each $|\zeta_i|\rightarrow \infty$, $i=1,\dots,n$.  Hence the above two integrals are convergent. Therefore when  the zero $\zeta_k$ is not on the contour, the bound \eqref{bdd3} and consequently the bound \eqref{bdd} holds true.

Now assume that $P_n(z)=\sum_{j=0}^n\nu_j p_j(z)$,
where the $\nu_j$'s are complex Gaussian random variables and the $p_j$'s are polynomials such that
$p_j(z)=\sum_{k=0}^ja_{j,k}z^k$
with $a_{j,k}\in \C$ for $j=0,1,\dots, n$ and $k=0,1\dots,j$.
Making the following change of variables
\begin{align*}
\eta_n&=\nu_na_{n,n}\\
\eta_{n-1}&=\nu_na_{n,n-1}+\nu_{n-1}a_{n-1,n-1}\\
\eta_{n-2}&=\nu_na_{n,n-2}+\nu_{n-1}a_{n-1,n-2}+\nu_{n-2}a_{n-2,n-2} \\
&\ \ \vdots \\
\eta_0&=\nu_na_{n,0}+\nu_{n-1}a_{n-1,0}+\cdots+\nu_0a_{0,0},
\end{align*}
and calling the map which makes this change of variables $\psi$, we see that the Jacobian determinant for this change of variables is $|\det D\psi|^2=\prod_{j=0}^n|a_{j,j}|^2$.

Setting $\vec{\nu}=(\nu_n,\dots, \nu_0)$, so that
$\psi(\vec{\nu})
=(\eta_n,\dots,\eta_0)
=\vec{\eta}$, and using the change of variables formula we have
\begin{align*}
 \int_{\psi(\R^{2(n+1)})}\int_{\partial \Omega}\left|\frac{P_n^{\prime}(z,\vec{\nu})}{P_n(z,\vec{\nu})}\right| \ |dz|\ f(\vec{\nu})\  &dV \\ &=\int_{\R^{2(n+1)}}\int_{\partial \Omega}\left|\frac{P_n^{\prime}(z,\psi(\vec{\nu}))}{P_n(z,\psi(\vec{\nu}))}\right| \ |dz|\ f(\psi(\vec{\nu}))\ \prod_{j=0}^n|a_{j,j}|^2\ dV \\
&=\prod_{j=0}^n|a_{j,j}|^2\int_{\R^{2(n+1)}}\int_{\partial \Omega}\left|\frac{P_n^{\prime}(z,\vec{\eta})}{P_n(z,\vec{\eta})}\right| \ |dz|\ f(\vec{\eta})\ \ dV<\infty,
\end{align*}
since $\prod_{j=0}^n|a_{j,j}|^2<\infty$, and by our previous calculations.

Therefore by the above equality being finite, we are justified in using the complex version of the Fubini-Tonelli Theorem to exchange the expectation and the contour integral when
$P_n(z)=\sum_{j=0}^n\nu_j p_j(z)$.
\end{proof}

To prove Theorem \ref{thm2.1} we need one more lemma.
\begin{lemma}\label{lem3.2}
Let $F$ denote the function defined by \eqref{F}.  For $z\not\in \{z:K_{n}(z,z)=0\}$, we have
$$F(z)=\frac{K_{n}^{(0,1)}(\bar{z},\bar{z})}{K_{n}(z,z)}.$$
\end{lemma}
\begin{proof}
Observe that $P_n(z)=\sum_{j=0}^{n}\eta_jf_j(z)$ and $P_n^{'}(z)=\sum_{j=0}^{n}\eta_jf_j{'}(z)$, with $\eta_j=\alpha_j+i\beta_j$, are both complex Gaussian random variables.  We will work their real and imaginary parts
$$\text{Re}\ P_n(z)=\sum_{j=0}^n (\alpha_ja_j-\beta_jb_j):=\xi_1, \ \ \ \text{Im}\ P_n(z)=\sum_{j=0}^n (\alpha_jb_j+\beta_ja_j):=\xi_2, $$
$$\text{Re}\ P_n^{\prime}(z)=\sum_{j=0}^n (\alpha_jc_j-\beta_jd_j):=\xi_3, \ \ \ \text{Im}\ P_n^{\prime}(z)=\sum_{j=0}^n (\alpha_jd_j+\beta_jc_j):=\xi_4, $$
where
$$a_j=\text{Re } f_j(z),\ \ \ b_j=\text{Im } f_j(z),\ \ \ c_j=\text{Re } f_j^{'}(z),\ \ \ d_j=\text{Im } f_j^{'}(z).$$

Following Vanderbei's method proof, we will now be forming the covariance matrix of the vector
$\xi:=(\xi_1, \xi_2, \xi_3, \xi_4 )^T.$  Before doing so, observe that since the random variables $\{\alpha_j\}_{j=0}^n$ and $\{\beta_j\}_{j=0}^n$ are independent identically distributed $N(0,1)$, for $j=0,1,\dots,n$ we have that  $\E[\alpha_j]=\E[\beta_j]=0$ and  $\E[\alpha_j^2]=\E[\beta_j^2]=1$. Thus by the expectation being linear, it follows that $\E[\xi_1]=\E[\xi_2]=\E[\xi_3]=\E[\xi_4]=0$.  Consequently each entry in the covariance matrix for $\xi$ is of the form $\E[(\xi_i-\E[\xi_i])(\xi_k-\E[\xi_k])]=\E[\xi_i \xi_k]$, where $i,k=1,\dots, 4$.
Using these observations, by definition of the covariance matrix  we see that
\begin{align}\label{CV}
\text{Cov}[\xi]=\E \left[\xi \xi^T\right]
=\begin{bmatrix}
\E[\xi_1^2]     & \E[\xi_1\xi_2]   & \E[\xi_1\xi_3]    & \E[\xi_1\xi_4] \\
\E[\xi_2 \xi_1]     & \E[\xi_2^2] & \E[\xi_2 \xi_3]  & \E[\xi_2 \xi_4] \\
\E[\xi_3 \xi_1]     & \E[\xi_3 \xi_2] & \E[\xi_3^2]  & \E[\xi_3 \xi_4]\\
\E[\xi_4 \xi_1]      & \E[\xi_4 \xi_2]   & \E[\xi_4 \xi_3]  & \E[\xi_4^2]
\end{bmatrix}.
\end{align}

We now represent the correlated Gaussian random variables $\xi_1,\xi_2, \xi_3,\xi_4$ in terms of four independent standard normals by finding a lower triangular matrix $L=[l_{pq}]$, $p, q=1,2,3,4$, with the property that  $\xi\,{\buildrel d \over =}\,L\zeta$, where the notation ${\buildrel d \over =}$ denotes equality in distribution, with $\zeta=(\zeta_1,\zeta_2,\zeta_3,\zeta_4)^T$ being a vector of four independent standard normal random variables.  Since
\begin{equation}\label{CV1}
\text{Cov}[\xi]=\E\left[\xi \xi^T\right]=\E\left[L\zeta \zeta^T L^T\right]=LL^T,
\end{equation}
we see that $L$ is the Cholesky factor for the covariance matrix.

By $\xi\,{\buildrel d \over =}\,L\zeta$, and the fact that $L$ is lower triangular, we have
\begin{align*}
\frac{P_n'(z)}{P_n(z)}&=\frac{\xi_3+i\xi_4}{\xi_1+i\xi_2}
\,{\buildrel d \over =}\,\frac{(l_{31}+il_{41})\zeta_1+(l_{32}+il_{42})\zeta_2+(l_{33}+il_{43})\zeta_3+il_{44}\zeta_4}{(l_{11}+il_{21})\zeta_1+il_{22}\zeta_2}.
\end{align*}

So with $\alpha=l_{31}+il_{41},\beta=l_{32}+il_{42},\gamma=l_{11}+il_{21}$, and $\delta=il_{22}$, using the independence of the $\zeta_i$'s, it follows that
$$
\E\left[\frac{P_n'(z)}{P_n(z)}\right]=\E \left[ \frac{\alpha \zeta_1+\beta\zeta_2}{\gamma\zeta_1+\delta\zeta_2}\right].
$$
If we now split up the numerator of the above and use the property that $\zeta_1$ and $\zeta_2$ are exchangeable, we can write the expectation as
$$
F(z)=\E\left[\frac{P_n'(z)}{P_n(z)}\right]=\frac{\alpha}{\delta}f(\gamma /\delta)+\frac{\beta}{\gamma}f(\delta/\gamma),
$$
where $f:\C\setminus \R \rightarrow \C$ by $f(w)=\E\left[\frac{\zeta_1}{w\zeta_1+\zeta_2}\right].$
Using the definition of the expectation in our Gaussian setting, and appealing to polar integration we see that
\begin{align*}
f(w)&=\E\left[\frac{\zeta_1}{w\zeta_1+\zeta_2}\right]\\
&=\frac{1}{2\pi}\int_0^{2\pi}\int_0^{\infty}\frac{\rho \cos \theta}{w\rho \cos \theta+ \rho \sin \theta}e^{-\rho^2/2}\rho d\rho d\theta \\
&=\frac{1}{2\pi}\int_0^{2\pi}\frac{d\theta}{w+\tan \theta}\\
&=\begin{cases}
\frac{1}{w+i} \hfill &\text{ if Im}(w)>0,\\
\frac{1}{w-i} \hfill &\text{ if Im}(w)<0.
\end{cases}
\end{align*}
We need to evaluate $f$ at $\gamma / \delta$ and $\delta / \gamma$.  Since in general $l_{11}$ and $l_{22}$ are nonnegative, we have that $\gamma/ \delta=l_{21} / l_{22}-i l_{11}/ l_{22}$ has negative imaginary part while $\delta/ \gamma=l_{21}l_{22}/ (l_{11}^2+l_{21}^2)+il_{11}l_{22}/(l_{11}^2+l_{21}^2)$ has positive imaginary part. Thus
\begin{align}\label{3.2.1}
F(z)&=\frac{\alpha}{\delta}f(\gamma /\delta)+\frac{\beta}{\gamma}f(\delta/\gamma)
=\frac{\alpha}{\delta}\frac{1}{\frac{\gamma}{\delta}-i}+\frac{\beta}{\gamma}\frac{1}{\frac{\delta}{\gamma}+i}
=\frac{i\alpha+\beta}{i\gamma +\delta}
=\frac{l_{32}-l_{41}+i(l_{31}+l_{42})}{-l_{21}+i(l_{11}+l_{22})}.
\end{align}
From the above we see that we need explicit formulas for the elements of the Cholesky factor $L$. Using \eqref{CV} and \eqref{CV1}  we obtain
\begin{align*}
&\E[\xi_1^2]=l_{11}^2,\ \  \ \ \E[\xi_2 \xi_1]=l_{21}l_{11},\ \ \ \ \E[\xi_3 \xi_1]=l_{31}l_{11},\ \ \ \ \E[\xi_4 \xi_1]=l_{41}l_{11}, \\
& \E[\xi_2^2]=l_{21}^2+l_{22}^2, \ \ \ \ \E[\xi_3 \xi_2]=l_{31}l_{21}+l_{32}l_{22},\ \ \ \ \E[\xi_4 \xi_2]=l_{41}l_{21}+l_{42}l_{22}.
\end{align*}
To solve for the above entries, we will first find what expectations above are.  To this end, recall that the random variables $\{\alpha_j\}_{j=0}^n$ and $\{\beta_j\}_{j=0}^n$ are independent identically distributed $N(0,1)$; consequently  $\E[\omega \nu]=\E[\omega]\E[\nu]$ for $\omega$ and $\nu$ being any two distinct elements among the $\alpha_j$'s and the $\beta_j$'s. Using these properties and the definitions of the sums $K_{n}(z,z)$ and $K_{n}^{(0,1)}(\bar{z},\bar{z})$ yields
\begin{align*}
\E[\xi_1^2]&=\sum_{j=0}^n(a_j^2+b_j^2)=\sum_{j=0}^n|f_j(z)|^2=K_{n}(z,z), \\
\E[\xi_2 \xi_1]&=\sum_{j=0}^n(a_jb_j-a_jb_j)=0,\\
\E[\xi_3 \xi_1]& =\sum_{j=0}^n(a_jc_j+b_jd_j)=\frac{K_{n}^{(0,1)}(\bar{z},\bar{z})
+\overline{K_{n}^{(0,1)}(\bar{z},\bar{z})}}{2}
=\text{Re}(K_{n}^{(0,1)}(\bar{z},\bar{z})), \\
\E[\xi_4\xi_1]&=\sum_{j=0}^n(a_jd_j-b_jc_j)=\frac{K_{n}^{(0,1)}(\bar{z},\bar{z})
-\overline{K_{n}^{(0,1)}(\bar{z},\bar{z})}}{2i}
=\text{Im}(K_{n}^{(0,1)}(\bar{z},\bar{z})), \\
\E[\xi_2^2]&=\sum_{j=0}^n(b_j^2+a_j^2)=\sum_{j=0}^n|f_j(z)|^2
=K_{n}(z,z),\\
\E[\xi_3 \xi_2]&=\sum_{j=0}^n(b_jc_j-a_jd_j)=-\text{Im}(K_{n}^{(0,1)}(\bar{z},\bar{z})),\\
\text{and}& \\
\E[\xi_4 \xi_2]&=\sum_{j=0}^n(b_jd_j+a_jc_j)
=\text{Re}(K_{n}^{(0,1)}(\bar{z},\bar{z})).
\end{align*}

Using the above equalities and solving for the needed entries of the Cholesky factor $L$ it follows that
\begin{align*}
&l_{11}=\sqrt{K_{n}(z,z)},\ \ \ \ l_{21}=0,\ \ \ \ l_{31}=\frac{\text{Re}(K_{n}^{(0,1)}(\bar{z},\bar{z}))}{\sqrt{K_{n}(z,z)}},
\ \ \ \ l_{41}=\frac{\text{Im}(K_{n}^{(0,1)}(\bar{z},\bar{z}))}{\sqrt{K_{n}(z,z)}}\\
&l_{22}=\sqrt{K_{n}(z,z)},\ \ \ \ l_{32}=\frac{-\text{Im}(K_{n}^{(0,1)}(\bar{z},\bar{z}))}{\sqrt{K_{n}(z,z)}},
\ \ \ \  l_{42}=\frac{\text{Re}(K_{n}^{(0,1)}(\bar{z},\bar{z}))}{\sqrt{K_{n}(z,z)}}\ .
\end{align*}

Therefore substituting these expressions into \eqref{3.2.1} and simplifying gives
\begin{align*}
F(z)&=\frac{l_{32}-l_{41}+i(l_{31}+l_{42})}{-l_{21}+i(l_{11}+l_{22})}\\
&=\frac{\frac{-\text{Im}(K_{n}^{(0,1)}(\bar{z},\bar{z}))}{\sqrt{K_{n}(z,z)}}
-\frac{\text{Im}(K_{n}^{(0,1)}(\bar{z},\bar{z}))}{\sqrt{K_{n}(z,z)}}+
i\left(\frac{\text{Re}(K_{n}^{(0,1)}(\bar{z},\bar{z}))}{\sqrt{K_{n}(z,z)}}+ \frac{\text{Re}(K_{n}^{(0,1)}(\bar{z},\bar{z}))}{\sqrt{K_{n}(z,z)}} \right)}{0+i\left(\sqrt{K_{n}(z,z)}+\sqrt{K_{n}(z,z)}  \right)} \\
&=\frac{\text{Re}(K_{n}^{(0,1)}(\bar{z},\bar{z}))+i\text{Im}(K_{n}^{(0,1)}(\bar{z},\bar{z}))}{K_{n}^{(0,1)}(\bar{z},\bar{z})}\\
&=\frac{K_{n}^{(0,1)}(\bar{z},\bar{z})}{K_{n}(z,z)}.
\end{align*}
\end{proof}

We now give the proof of Theorem \ref{thm2.1}.
\begin{proof}[Proof of Theorem \ref{thm2.1}]
Let us first observe that since $K_n(z,z)$ is the sum of the modulus squared of entire functions $f_j(z)$, $j=0,1,\dots, n$, we have $K_n(z,z)=0$  if and only if all the $f_j$'s vanish at a point.  Furthermore, since on every bounded region $K_n(z,z)$ can have only finitely many zeros, we can therefore avoid those zeros.  Thus to prove the theorem it suffices to consider a bounded region $\Omega$ such that $K_{n}(z,z)\neq 0$ for all $z\in \Omega$.    Setting $R=\{z\in \C: K_{n}(z,z)=0\}$, for a region $\Omega$ with $\Omega \cap R=\emptyset$, by Green's Theorem we have
$$\E[N_n(\Omega)]=\frac{1}{2\pi i}\int_{\partial \Omega}F(z)\ dz
=\frac{1}{\pi }\int_{\Omega} \frac{\partial F(z,\overline{z})}{\partial\overline{z}} \, dx\,dy,
$$
where are writing $F(z,\bar{z})$ to emphasize that $F$ is a function of both $z$ and $\bar{z}$.

Let the symbol $\overline{\partial}$ denote partial derivatives with respect to $\overline{z}$.
Our goal is to simplify $\frac{1}{\pi}\overline{\partial}F(z,\bar{z})$ to $h_n(z)$, where by Lemma \ref{lem3.2},
$$
F(z,\bar{z})=\frac{K_{n}^{(0,1)}(\bar{z},\bar{z})}{K_{n}(z,z)}.
$$
Using the Quotient Rule we see that
\begin{align*}
\overline{\partial} F(z,\bar{z})&= \frac{\overline{\partial}K_{n}^{(0,1)}(\bar{z},\bar{z})K_{n}(z,z)-K_{n}^{(0,1)}(\bar{z},\bar{z})\overline{\partial}K_{n}(z,z)}
{\left(K_{n}(z,z)\right)^2} \\
&=\frac{K_{n}^{(1,1)}(z,z)K_{n}(z,z)-K_{n}^{(0,1)}(\bar{z},\bar{z})K_{n}^{(0,1)}(z,z)}
{\left(K_{n}(z,z)\right)^2}\\
&=\frac{K_{n}^{(1,1)}(z,z)K_{n}(z,z)-|K_{n}^{(0,1)}(z,z)|^2}{\left(K_{n}(z,z)\right)^2}.
\end{align*}

Therefore
$$h_n(z)=\frac{1}{\pi}\ \overline{\partial} F(z,\bar{z})=\frac{K_{n}^{(1,1)}(z,z)K_{n}(z,z)-|K_{n}^{(0,1)}(z,z)|^2}{\pi \left(K_{n}(z,z)\right)^2}.$$
\end{proof}

\section{Random Polynomials spanned by OPRL}

In this section we derive the intensity function $h_n^{P}(z)$  for the random orthogonal polynomial \eqref{ROS}.
Since the polynomials $p_j(z)$, $j=0,1,\dots,n$,  are orthogonal on the real line, we can simplify the kernels $K_{n}(z,z)$, $K_{n}^{(0,1)}(z,z)$, and $K_{n}^{(1,1)}(z,z)$ which make up the intensity function $h_n^{P}$ using the Christoffel-Darboux formula.  For convenience of the reader, the Christoffel-Darboux formula  (p. 43 of \cite{SZ}) says that for $z,w\in \C$ and $p_n(z)$, $n=0,1,\dots$,  polynomials that are orthogonal on the real line, and $k_n$ be the leading coefficient of $p_n(z)$, we have
\begin{equation}\label{CD1}
\sum_{j=0}^{n}p_j(z)p_j(w)=\frac{k_{n}}{k_{n+1}}\cdot \frac{p_{n+1}(z)p_{n}(w)-p_{n}(z)p_{n+1}(w)}{z-w}, \ \  z\neq w.
\end{equation}
Furthermore, taking $z=w$ gives
\begin{equation}\label{CD2}
\sum_{j=0}^{n}\left(p_j(z)\right)^2=\frac{k_{n}}{k_{n+1}}\cdot (p_{n+1}^{\prime}(z)p_{n}(z)-p_{n}^{\prime}(z)p_{n+1}(z)).
\end{equation}

Before obtaining our representations of the kernels, let us note that since the polynomials $p_j(z)$, $j=0,1,\dots,n$,  are orthogonal on the real line, they consequently have real coefficients.  Thus when using conjugation we have that $\overline{p_j(z)}=p_j(\bar{z})$ for all $j=0,1,\dots,n $ and all $z\in \C$.

Using the Christoffel-Darboux formula to get a representation for $K_{n}(z,z)$, we take $w=\bar{z}$ in \eqref{CD1} to achieve
\begin{align}\nonumber
K_{n}(z,z)&=\sum_{j=0}^{n}p_j(z)\overline{p_j(z)} \\
&=\frac{k_{n}}{k_{n+1}}\cdot \frac{p_{n+1}(z)p_{n}(\bar{z})-p_{n}(z)p_{n+1}(\bar{z})}{2i\text{Im}(z)}  \label{CDB01} \\
&=\frac{k_{n}}{k_{n+1}}\cdot \frac{\text{Im}(p_{n+1}(z)p_{n}(\bar{z}))}{\text{Im}(z)}. \label{CDB02}
\end{align}

For our representation of $K_{n}^{(0,1)}(z,z)$, we first take the derivative of \eqref{CD1} with respect to $w$.  This gives
\begin{align}\label{CDw}
\sum_{j=0}^{n}p_j(z)p_j^{\prime}(w)&=\frac{k_{n}}{k_{n+1}} \left(\frac{p_{n+1}(z)p_{n}^{\prime}(w)-p_{n}(z)p_{n+1}^{\prime}(w)}{z-w}
   +\frac{p_{n+1}(z)p_{n}(w)-p_{n}(z)p_{n+1}(w) }{(z-w)^2}\right).
\end{align}
Setting $w=\bar{z}$ in the above, and using our representation of $K_{n}(z,z)$ at \eqref{CDB01} we then have
\begin{align}\label{CDB1}
K_{n}^{(0,1)}(z,z)&=\sum_{j=0}^{n} p_j(z)\overline{p_j^{\prime}(z)}
=\frac{k_{n}}{k_{n+1}} \Bigg[\frac{p_{n+1}(z)p_{n}^{\prime}(\bar{z})-p_{n}(z)p_{n+1}^{\prime}(\bar{z})}{2i\text{Im}(z)}\Bigg]+\frac{K_{n}(z,z) }{2i\text{Im}(z)}.
\end{align}

Since
$$\overline{K_{n}^{(0,1)}(z,z)}=\overline{\sum_{j=0}^{n} p_j(z)\overline{p_j^{\prime}(z)}}=\sum_{j=0}^{n} \overline{p_j(z)}p_j^{\prime}(z)=\sum_{j=0}^{n} p_j(\bar{z})p_j^{\prime}(z)=K_{n}^{(0,1)}(\bar{z},\bar{z}),$$
taking the conjugate of \eqref{CDB1} yields
\begin{align}
K_{n}^{(0,1)}(\bar{z},\bar{z})=\frac{k_{n}}{k_{n+1}} \Bigg[\frac{p_{n}(\bar{z})p_{n+1}^{\prime}(z)-p_{n+1}(\bar{z})p_{n}^{\prime}(z)}{2i\text{Im}(z)}\Bigg]-\frac{K_{n}(z,z) }{2i\text{Im}(z)}. \label{CDB11}
\end{align}

To obtain our representation for $K_{n}^{(1,1)}(z,z)$, we differentiate \eqref{CDw} with respect to $z$, then
taking $w=\bar{z}$, and recalling our representation for $K_{n}(z, z)$ at \eqref{CDB01} and then that of $K_{n}^{(0,1)}(z,z)$ from \eqref{CDB1} and $K_{n}^{(0,1)}(\bar{z},\bar{z})$ from \eqref{CDB11}, we see that
\begin{align}\label{CDB2}
K_{n}^{(1,1)}(z,z)&=\sum_{j=0}^{n}p_j^{\prime}(z)\overline{p_j^{\prime}(z)}
=\frac{K_{n}^{(0,1)}(\bar{z},\bar{z})}{2i\text{Im}(z)}-\frac{K_{n}^{(0,1)}(z,z)}{2i\text{Im}(z)}+ \frac{k_{n}}{k_{n+1}}\cdot\frac{\text{Im}(p_{n+1}^{\prime}(z)p_{n}^{\prime}(\bar{z}))}{\text{Im}(z)}.
\end{align}

For our representation of $K_{n}(z,\bar{z})$ we simply use \eqref{CD2} to achieve
\begin{align}\label{CDA0}
K_{n}(z,\bar{z})&=\sum_{j=0}^{n}p_j(z)\overline{p_j(\bar{z})}=\sum_{j=0}^{n}p_j(z)p_j(z)
=\frac{k_{n}}{k_{n+1}}\left(p_{n+1}^{\prime}(z)p_{n}(z)-p_{n}^{\prime}(z)p_{n+1}(z)\right).
\end{align}

Using our derived expressions \eqref{CDB01}, \eqref{CDB1}, \eqref{CDB11}, \eqref{CDB2}, and \eqref{CDA0}, the numerator of the intensity function from Theorem \ref{thm2.1} simplifies as
\begin{align}
\nonumber
K_{n}^{(1,1)}&(z,z)K_{n}(z,z)-|K_{n}^{(0,1)}(z,z)|^2
=\frac{\left(K_{n}(z,z) \right)^2-\left| K_{n}(z,\bar{z}) \right|^2}{4\left(\text{Im}(z)\right)^2}.
\end{align}

Therefore, using the expression for the numerator above and recalling the relations \eqref{CDB02} and \eqref{CDA0}, we see that
\begin{align}
\nonumber
h_n^P(z)&=\frac{K_{n}^{(1,1)}(z,z)K_{n}(z,z)-|K_{n}^{(0,1)}(z,z)|^2}{\pi  \left(K_{n}(z,z)\right)^2} \\
\label{KNN}
&=\frac{1}{4\pi\left(\text{Im}(z)\right)^2}\left(1-\frac{\left|K_{n}(z,\bar{z})\right|^2}{\left(K_{n}(z,z)\right)^2}\right)\\
\nonumber
&=\frac{1}{4\pi\left(\text{Im}(z)\right)^2}-\frac{|p_{n+1}^{\prime}(z)p_n(z)-p_n^{\prime}(z)p_{n+1}(z)|^2}{4\pi \left(\text{Im}(p_{n+1}(z)p_n(\bar{z})  \right)^2},
\end{align}
which gives the result of Theorem \ref{OP}.

\section{The Limiting Intensity Function for OPRL associated to the Szeg\H{o} Class}

To prove Theorem \ref{GOP2}, by \eqref{KNN} we see that it suffices to find asymptotics for the kernels $K_{n}(z,z)$ and $K_{n}(z,\bar{z})$.

Since $w(x)$ is a weight function on the interval $-1\leq x\leq 1$ such that $w(\cos \theta)|\sin \theta |=f(\theta)$ belongs to the weight class $G$, it is know that associated to the function $f(\theta)$ is a uniquely determined function
$$D(f;\xi)=D(\xi)=\exp\left\{ \frac{1}{4\pi}\int_{-\pi}^{\pi}\log f(t) \ \frac{1+\xi e^{-it}}{1-\xi e^{-it}}\ dt \right\},$$
that is analytic and nonzero for $|\xi|<1$ with $D(0)>0$.  Using Szeg\H{o}'s Theorem 12.1.2 on page 297 of \cite{SZ},  for $D(\xi)$ as above, we have that the orthogonal polynomials $\{p_n(z)\}$ associated with $w(x)$ have the asymptotic formula
\begin{equation}\label{GOP1}
p_n(z)\sim \frac{\xi^n}{\sqrt{2\pi}D(\xi^{-1})},
\end{equation}
where $z=\frac{1}{2}(\xi+\xi^{-1})\in \C\setminus [-1,1]$, with $|\xi|>1$.  Moreover, \eqref{GOP1} holds uniformly for $|\xi|\geq R>1$.

Let $\{\phi_{2n}(\xi)\}$ be a subsequence of $\{\phi_n(\xi)\}$ the orthonormal set associated with $f(\theta)$ on $\xi=e^{i\theta}$, and let $\kappa_{2n}$ be the leading coefficient of $\phi_{2n}$. Then the asymptotic \eqref{GOP1} is a consequence of the following results:  equation (11.5.2) on page 294 in \cite{SZ} that gives
\begin{align}\label{pz}
p_n(z)&=(2\pi)^{-\frac{1}{2}}\left(1+\frac{\phi_{2n}(0)}{\kappa_{2n}}\right)^{-\frac{1}{2}}\left(\xi^{-n}\phi_{2n}(\xi)+\xi^n\phi_{2n}(\xi^{-1})\right),
\end{align}
setting $n$ to be $2n$ in Theorem 12.1.1 \footnote{ We note that since $f(\theta)=w(\cos \theta)|\sin \theta|$ is an even function, in the result of Theorem 12.1.1 on page 297 of \cite{SZ} we have $\bar{D}(\xi^{-1})=D(\xi^{-1})$.} on page 297 of \cite{SZ}, which yields
\begin{equation}\label{p1}
\lim_{n\rightarrow \infty}\frac{\phi_{2n}(\xi)D(\xi^{-1})}{\xi^{2n}}=1
\end{equation}
uniformly for $|\xi|\geq R >1$, as well as uniform limits
\begin{align}
\lim_{n\rightarrow \infty}\phi_{2n}(\xi^{-1})=0,\ \ \ \lim_{n\rightarrow \infty}\phi_{2n}(0)=0, \ \ \text{and}\ \ \   \lim_{n\rightarrow \infty}\kappa_{2n}=\kappa>0,\label{plims1}
\end{align}
from page 304 of \cite{SZ}.

For an asymptotic for the kernel $K_{n}(z,\bar{z})$ we will need an asymptotic for the derivative of the orthogonal polynomial $p_n(z)$.  Since $z=\frac{1}{2}(\xi+\xi^{-1})$, where $|\xi|>1$, when we solve for $\xi$ we have $\xi=z+\sqrt{z^2-1}$. Hence differentiating  \eqref{pz} with respect to $z$ yields
\begin{align}\label{pnprime}
\nonumber
\frac{d}{dz}p_n(z)=&(2\pi)^{-\frac{1}{2}}\left(1+\frac{\phi_{2n}(0)}{\kappa_{2n}}\right)^{-\frac{1}{2}} \\
& \ \ \cdot\frac{-n\xi^{-n}\phi_{2n}(\xi)+\xi^{-n+1}\phi_{2n}^{\prime}(\xi)  +n\xi^n\phi_{2n}(\xi^{-1})-\xi^{n-1}\phi_{2n}^{\prime}(\xi^{-1})    }{\sqrt{z^2-1}}.
\end{align}
For cancelation purposes our asymptotic for $p_n^{\prime}(z)$ will be in terms asymptotics of the functions $\phi_{2n}$ and their derivatives $\phi_{2n}^{\prime}$.  Since we already know the asymptotics of $\phi_{2n}(\xi)$ and $\phi_{2n}(\xi^{-1})$, by inspection of \eqref{pnprime} we see that we need asymptotics for $\phi_{2n}^{\prime}(\xi)$ and $\phi_{2n}^{\prime}(\xi^{-1})$ to complete the desired asymptotic of $p_n^{\prime}(z)$.

To obtain an asymptotic for $\phi_{2n}^{\prime}(\xi)$, observe that from \eqref{p1} we have
\begin{equation}\label{p6}\phi_{2n}(\xi)D(\xi^{-1})\xi^{-2n}-1=e_{n}(\xi),
\end{equation}
where $e_n(\xi)\rightarrow 0$ uniformly for $|\xi|\geq R>1$.  Furthermore, since the left-hand side of the above is holomorphic on $\C\setminus \overline{\D}$, we have that $e_n(\xi)$ is holomorphic $\C\setminus \overline{\D}$.  Thus differentiating  \eqref{p6} with respect to $\xi$ yields
\begin{align}\label{p4}
\frac{\phi_{2n}^{\prime}(\xi)D(\xi^{-1})}{\xi^{2n}}-\frac{\phi_{2n}(\xi)D^{\ \prime}(\xi^{-1})}{\xi^{2n+2}}-\frac{2n\phi_{2n}(\xi)D(\xi^{-1})}{\xi^{2n+1}} = e_n^{\prime}(\xi).
\end{align}

We seek to show that $\lim_{n\rightarrow \infty}e_n^{\ \prime}(\xi)=0$ uniformly on compact subsets of $\{\xi:|\xi|\geq r >1\}$.
Let us first make the claim that $\lim_{\xi \rightarrow \infty}e_n(\xi)$ exists.  From this limit we will have that not only is $e_n(\xi)$ holomorphic on $\C\setminus \overline{\D}$, but it is actually holomorphic on $\overline{\C}\setminus \overline{\D}$ for all $n=0,1,\dots$, in the extended complex plane.  To see that the claim is true, notice that
\begin{align*}
\lim_{\xi \rightarrow \infty}e_n(\xi)&=
\lim_{\xi \rightarrow \infty}\left[\left(\kappa_{2n}\xi^{2n}+\kappa_{2n-1}\xi^{2n-1}+\cdots +\kappa_1\xi+\kappa_0\right)D(\xi^{-1})\xi^{-2n}-1\right] \\
&=  \left\{
        \begin{array}{ll}
            \kappa_0D(0)-1, & \quad n = 0, \\
            \kappa_{2n}D(0)-1, & \quad n > 0.
        \end{array}
    \right.
\end{align*}
Since both $\kappa_0D(0)-1$ and $\kappa_{2n}D(0)-1$ exits and are finite, we have that the $\lim_{\xi\rightarrow \infty}e_n(\xi)$ exits for all $n=0,1\dots$, and consequently $e_n(\xi)$ is holomorphic at infinity for all $n\in \{0,1\dots\}$.

Given that $e_n(\xi)$ is holomorphic on $\overline{\C}\setminus \overline{\D}$, we can apply Cauchy's Formula for unbounded domains to $e_n(\xi)$ on the unbounded set $\{\xi:|\xi|\geq r >1\}$. Let $K\subset \{\xi:|\xi|\geq r >1\}$ be a compact set with $z\in K$.  Using this theorem we have
$$e_n(z)=\frac{-1}{2\pi i}\int_{\partial D(0,r)}\frac{e_n(t)}{t-z} \ dt +\lim_{z \rightarrow \infty} e_n(z).$$
Taking the derivative of the above with respect to $z$ and estimating the integral yields
\begin{align*}
|e_n^{\prime}(z)|&
\leq \frac{1}{2\pi }\int_{\partial D(0,r)}\left|\frac{e_n(t)}{(t-z)^2}\right| \ d|t|
\leq \frac{r}{d^2}\ \displaystyle\max_{t\in \partial D(0,r)}|e_n(t)|,
\end{align*}
where $d=\text{Dist}(K,\partial D(0,r))>0$ since $K$ is a compact subset of $\{\xi:|\xi|\geq r >1\}$.
Under the hypothesis that $e_n(z)$ converges uniformly to zero as $n\rightarrow \infty$ for $z\in\{\xi:|\xi|\geq r >1\}$, and in particular all $t   \in \partial D(0,r)$, we have
$$\lim_{n\rightarrow \infty}|e_n^{\ \prime}(z)|\leq \lim_{n\rightarrow \infty}\frac{r}{d^2}\ \displaystyle\max_{t\in \partial D(0,r)}|e_n(t)|=0,$$
where the limit is uniform in $z$ as $n\rightarrow \infty$.

Therefore by $K$ being an arbitrary compact subset of $\{\xi:|\xi|\geq r >1\}$, it follows that $\lim_{n\rightarrow \infty}e_n^{\ \prime}(\xi)=0$ uniformly on compact subsets of $\{\xi:|\xi|\geq r >1\}$.

Taking the above into account and solving \eqref{p4} for $\phi_{2n}^{\prime}(\xi)\xi^{-2n}$ we have
\begin{align}\label{p5}
\frac{\phi_{2n}^{\prime}(\xi)}{\xi^{2n}}=\frac{\phi_{2n}(\xi)D^{\ \prime}(\xi^{-1})}{\xi^{2n+2}D(\xi^{-1})}+\frac{2n\phi_{2n}(\xi)}{\xi^{2n+1}} +o(1).
\end{align}

For our asymptotic of $\phi_{2n}^{\prime}(\xi^{-1})$ we begin with equation (12.3.17) on page 303 in \cite{SZ} that says
$$\sum_{j=0}^{\infty}\overline{\phi_j(a)}\phi_j(w)=\frac{1}{(1-\bar{a}w)\overline{D(a)}D(w)}<\infty,$$
for $|a|<1$ and $|w|<1$.  Since the above sum  of analytic functions in $w$ and of anti-holomorphic functions in $\bar{a}$ converges uniformly on compact subsets of the unit disk, we can differentiate it with respect to $w$ and $\bar{a}$ on compact subsets of the unit disk and retain the uniform convergence of the derivatives of the infinite sum on compact subsets of the unit disk.  Taking these derivatives and setting $w=a=\xi^{-1}$ we achieve $\sum_{j=0}^{\infty}|\phi_j^{\prime}(\xi^{-1})|^2<\infty$.  Consequently on compact subsets of $\{\xi^{-1}:|\xi^{-1}|<1\}$ uniformly we have
\begin{equation}\label{phiprimeinv}
\lim_{n\rightarrow \infty}\phi_n^{\prime}(\xi^{-1})=0.
\end{equation}

We now have all the needed asymptotics to give the proof of Theorem \ref{GOP2}.

\begin{proof}[Proof of Theorem \ref{GOP2}]
To obtain an asymptotic expression for $K_{n}(z,z)$, we first note that since $f(\theta)$ is an even function, we have that $\phi_{2n}$ and $\phi_{2(n+1)}$ have real coefficients.  Thus $\phi_{2n}(0)=\overline{\phi_{2n}(0)}$ and $\phi_{2(n+1)}(0)=\overline{\phi_{2(n+1)}(0)}$, as well as $\kappa_{2n}=\overline{\kappa_{2n}}$ and  $\kappa_{2(n+1)}=\overline{\kappa_{2(n+1)}}$.  Taking this into consideration and using the Christoffel-Darboux formula and then the representation \eqref{pz} yields
\begin{align}
\nonumber
K_{n}(z,z)&=\sum_{j=0}^{n}p_j(z)p_j(\bar{z}) \\
\nonumber
&=\frac{k_{n}}{k_{n+1}}\cdot \frac{p_{n+1}(z)p_{n}(\bar{z})-p_{n}(z)p_{n+1}(\bar{z})}{2i\text{Im}(z)}   \\
\nonumber
&=\frac{k_{n}}{k_{n+1}2i\text{Im}(z)}(2\pi)^{-1}\left(1+\frac{\phi_{2(n+1)}(0)}{\kappa_{2(n+1)}}\right)^{-\frac{1}{2}}
\left(1+\frac{\phi_{2n}(0)}{\kappa_{2n}}\right)^{-\frac{1}{2}}\\
\nonumber
&\ \ \ \cdot \Bigg[
\left(\xi^{-(n+1)}\phi_{2(n+1)}(\xi)
+\xi^{n+1}\phi_{2(n+1)}(\xi^{-1})\right)\left(\overline{\xi^{-n}\phi_{2n}(\xi)}
+\overline{\xi^{n}\phi_{2n}(\xi^{-1})}\right)\\
\nonumber
&\ \ \ \ \ \ \ - \left(\xi^{-n}\phi_{2n}(\xi)
+\xi^{n}\phi_{2n}(\xi^{-1})\right)\left(\overline{\xi^{-(n+1)}\phi_{2(n+1)}(\xi)}
+\overline{\xi^{n+1}\phi_{2(n+1)}(\xi^{-1})}\right)\Bigg]    \\
\nonumber
&=\frac{k_{n}}{k_{n+1}2i\text{Im}(z)}(2\pi)^{-1}\left(1+\frac{\phi_{2(n+1)}(0)}{\kappa_{2(n+1)}}\right)^{-\frac{1}{2}}
\left(1+\frac{\phi_{2n}(0)}{\kappa_{2n}}\right)^{-\frac{1}{2}}\\
&\ \ \ \cdot \Bigg[\frac{\xi \phi_{2(n+1)}(\xi)}{\xi^{n+2}}\frac{\overline{\phi_{2n}(\xi)}}{\overline{\xi^n}} -\frac{\overline{\xi} \phi_{2n}(\xi)}{\xi^{n}}\frac{\overline{\phi_{2(n+1)}(\xi)}}{\overline{\xi^{n+2}}} \label{kz1}\\
&\ \ \ \ \ \ \ +\frac{\xi \phi_{2(n+1)}(\xi)}{\xi^{n+2}}\overline{\xi^{n}\phi_{2n}(\xi^{-1})} - \frac{ \phi_{2n}(\xi)}{\xi^{n}}\overline{\xi^{n+1}\phi_{2(n+1)}(\xi^{-1})} \label{kz2}\\
&\ \ \ \ \ \ \ +\xi^{n+1}\phi_{2(n+1)}(\xi^{-1})\frac{\overline{\phi_{2n}(\xi)}}{\overline{\xi^{n}}} -\xi^{n}\phi_{2n}(\xi^{-1})\frac{\overline{\xi \phi_{2(n+1)}(\xi)}}{\overline{\xi^{n+2}}} \label{kz3} \\
&\ \ \ \ \ \ \ +|\xi|^{2n}\xi\phi_{2(n+1)}(\xi^{-1})\overline{\phi_{2n}(\xi^{-1})}- |\xi|^{2n}\overline{\xi}\phi_{2n}(\xi^{-1})\overline{\phi_{2(n+1)}(\xi^{-1})}\Bigg].\label{kz4}
\end{align}
To find the asymptotic of the above, let us first focus on the terms within the large brackets.  Observe that for the expression \eqref{kz1}, using \eqref{p1} we have
\begin{equation}\label{kz5}
\lim_{n\rightarrow \infty} \frac{ \frac{\xi \phi_{2(n+1)}(\xi)}{\xi^{n+2}}\frac{\overline{\phi_{2n}(\xi)}}{\overline{\xi^n}} -\frac{\overline{\xi} \phi_{2n}(\xi)}{\xi^{n}}\frac{\overline{\phi_{2(n+1)}(\xi)}}{\overline{\xi^{n+2}}} }{ 2i \text{Im}(\xi)\left| \frac{\xi^n}{D(\xi^{-1})} \right|^2 }=1.
\end{equation}
Also note that using \eqref{p1} and first limit in \eqref{plims1}, the expressions \eqref{kz2}, \eqref{kz3}, and \eqref{kz4} are all $o(|\xi|^{2n})$. Combining this with \eqref{kz5} and appealing to the second two limits of \eqref{plims1} we achieve
\begin{align*}
K_{n}(z,z)=\frac{k_{n}}{k_{n+1}} \frac{\text{Im}(\xi)}{2\pi\text{Im}(z)}\left| \frac{\xi^n}{D(\xi^{-1})} \right|^2+o(|\xi|^{2n}).
\end{align*}
Thus
\begin{equation}\label{Ksq}
\left( K_{n}(z,z) \right)^2=\left(\frac{k_{n}}{k_{n+1}} \frac{\text{Im}(\xi)}{2\pi\text{Im}(z)}\right)^2\left| \frac{\xi^n}{D(\xi^{-1})} \right|^4 +o\left(|\xi|^{4n}\right).
\end{equation}

For our representation of $K_{n}(z,\bar{z})$, using the Christoffel-Darboux formula and then \eqref{pz} and \eqref{pnprime} it follows that
\begin{align*}\nonumber
K_{n}(z,\bar{z})&=\sum_{j=0}^{n}p_j(z)p_j(z) \\
&=\frac{k_{n}}{k_{n+1}}\cdot(p_{n+1}^{\prime}(z)p_{n}(z)-p_{n}^{\prime}(z)p_{n+1}(z)) \\
&=\frac{k_{n}}{k_{n+1}}\frac{1}{2\pi \sqrt{z^2-1}}\left(1+\frac{\phi_{2(n+1)}(0)}{\kappa_{2(n+1)}}\right)^{-\frac{1}{2}}
\left(1+\frac{\phi_{2n}(0)}{\kappa_{2n}}\right)^{-\frac{1}{2}} \\
& \cdot \Bigg[\Big(-(n+1)\xi^{-(n+1)}\phi_{2(n+1)}(\xi)+\xi^{-(n+1)+1}\phi_{2(n+1)}^{\prime}(\xi) \\
&\ \ \ \ \ \ \ \  \ \ +(n+1)\xi^{n+1}\phi_{2(n+1)}(\xi^{-1})-\xi^{(n+1)-1}
\phi_{2(n+1)}^{\prime}(\xi^{-1}) \Big) \\
&\ \ \ \ \ \ \ \ \cdot\Big(\xi^{-n}\phi_{2n}(\xi)+\xi^n\phi_{2n}(\xi^{-1})\Big) \\
&\ \ \ \ -\Big(-n\xi^{-n}\phi_{2n}(\xi)+\xi^{-n+1}\phi_{2n}^{\prime}(\xi)
 +n\xi^{n}\phi_{2n}(\xi^{-1})-\xi^{n-1}
\phi_{2n}^{\prime}(\xi^{-1}) \Big) \\
&\ \ \ \ \  \ \ \ \ \cdot\Big(\xi^{-(n+1)}\phi_{2(n+1)}(\xi)+\xi^{n+1}\phi_{2(n+1)}(\xi^{-1})\Big)\Bigg].
\end{align*}
Expanding what is in the large brackets and collecting like terms gives the following:
\begin{align}
\label{k1}
&-\xi^{-2n-1}\phi_{2n}(\xi)\phi_{2(n+1)}(\xi) \\
\label{k2}
&+\xi^{-2n}\phi_{2n}(\xi)\phi_{2(n+1)}^{\prime}(\xi)-\xi^{-2n}\phi_{2n}^{\prime}(\xi)\phi_{2(n+1)}(\xi) \\
\label{k3}
&+(2n+1)\xi\phi_{2(n+1)}(\xi^{-1})\phi_{2n}(\xi)\\
\label{k4}
&-\phi_{2n}(\xi)\phi_{2(n+1)}^{\prime}(\xi^{-1})+\xi^{-2}\phi_{2(n+1)}(\xi)\phi_{2n}^{\prime}(\xi^{-1})\\
\label{k5}
&-(2n+1)\xi^{-1}\phi_{2n}(\xi^{-1})\phi_{2(n+1)}(\xi) \\
\label{k6}
&+\phi_{2n}(\xi^{-1})\phi_{2(n+1)}^{\prime}(\xi)-\xi^2\phi_{2(n+1)}(\xi^{-1})\phi_{2n}^{\prime}(\xi) \\
\label{k7}
&+\xi^{2n+1}\phi_{2n}(\xi^{-1})\phi_{2(n+1)}(\xi^{-1}) \\
\label{k8}
&-\xi^{2n}\phi_{2n}(\xi^{-1})\phi_{2(n+1)}^{\prime}(\xi^{-1})+\xi^{2n}\phi_{2(n+1)}(\xi^{-1})\phi_{2n}^{\prime}(\xi^{-1}).
\end{align}

Starting with easiest of the terms above, using \eqref{p1}, the first limit in \eqref{plims1}, and \eqref{phiprimeinv}, we have that \eqref{k4}, \eqref{k7}, and \eqref{k8} are all $o(\xi^{2n})$.

Using \eqref{p5} in  \eqref{k2} yields
\begin{align}
\nonumber
\xi^{-2n}\phi_{2n}&(\xi)\phi_{2(n+1)}^{\prime}(\xi)-\xi^{-2n}\phi_{2n}^{\prime}(\xi)\phi_{2(n+1)}(\xi) \\
\nonumber
&=\frac{\phi_{2n}(\xi)\phi_{2(n+1)}(\xi)D^{\ \prime}(\xi^{-1})}{\xi^{2(n+1)}D(\xi^{-1})}+\frac{2(n+1)\phi_{2n}(\xi)\phi_{2(n+1)}(\xi)}{\xi^{2n+1}}+\phi_{2n}(\xi)o(1)\\
\nonumber
&\ \ \ -\frac{\phi_{2n}(\xi)D^{\ \prime}(\xi^{-1})\phi_{2(n+1)}(\xi)}{\xi^{2(n+1)}D(\xi^{-1})}-\frac{2n\phi_{2n}(\xi)\phi_{2(n+1)}(\xi)}{\xi^{2n+1}}-\phi_{2(n+1)}(\xi)o(1)\\
\label{k9}
&=\frac{2\phi_{2n}(\xi)\phi_{2(n+1)}(\xi)}{\xi^{2n+1}}+o(\xi^{2n}),
\end{align}
where we have appealed to \eqref{p1} for the last equality.

Turning now to \eqref{k6}, again using \eqref{p5} gives us
\begin{align}
\nonumber
\phi_{2n}&(\xi^{-1})\phi_{2(n+1)}^{\prime}(\xi)-\xi^2\phi_{2(n+1)}(\xi^{-1})\phi_{2n}^{\prime}(\xi)\\
\nonumber
&=\frac{\phi_{2n}(\xi^{-1})\phi_{2(n+1)}(\xi)D^{\ \prime}(\xi^{-1})}{\xi^2D(\xi^{-1})}+\frac{2(n+1)\phi_{2n}(\xi^{-1})\phi_{2(n+1)}(\xi)}{\xi}
+\phi_{2n}(\xi^{-1})\xi^{2(n+1)}o(1)\\
\nonumber
&\ \ \ \ -\frac{\phi_{2(n+1)}(\xi^{-1})\phi_{2n}(\xi)D^{\ \prime}(\xi^{-1})}{D(\xi^{-1})}-2n\phi_{2(n+1)}(\xi^{-1})\phi_{2n}(\xi)\xi
-\xi^{2(n+1)}\phi_{2(n+1)}(\xi^{-1})o(1) \\
\label{k10}
&=\frac{2(n+1)\phi_{2n}(\xi^{-1})\phi_{2(n+1)}(\xi)}{\xi}-2n\xi \phi_{2(n+1)}(\xi^{-1})\phi_{2n}(\xi)+o(\xi^{2n}),
\end{align}
where are using \eqref{p1} to obtain the second equality.

Combining \eqref{k1}, \eqref{k9}, \eqref{k3}, \eqref{k5}, \eqref{k10}, and the fact that the other terms were $o(\xi^{2n})$,  for the terms from the large brackets we have
\begin{align}\nonumber
&\frac{\phi_{2n}(\xi)\phi_{2(n+1)}(\xi)}{\xi^{2n+1}}+\frac{\phi_{2n}(\xi^{-1})\phi_{2(n+1)}(\xi)}{\xi}
+\xi\phi_{2(n+1)}(\xi^{-1})\phi_{2n}(\xi)+o(\xi^{2n})\\
\label{b1}
&=\frac{\phi_{2n}(\xi)\phi_{2(n+1)}(\xi)}{\xi^{2n+1}}+o(\xi^{2n}),
\end{align}
by yet another time \eqref{p1}, and the first limit in \eqref{plims1}.

Thus using \eqref{b1}, appealing to \eqref{p1}, and the second two limits in \eqref{plims1}, we see that
\begin{equation*}
K_{n}(z,\bar{z})=\frac{k_{n}}{k_{n+1}}\frac{\xi^{2n+1}}{2\pi\sqrt{z^2-1}\left(D(\xi^{-1})\right)^2}+o(\xi^{2n}).
\end{equation*}

From above asymptotic we obtain
\begin{equation}\label{A0sq}
\left| K_{n}(z,\bar{z}) \right|^2=\left(\frac{k_{n}}{k_{n+1}}\right)^2\frac{|\xi^{2n+1}|^2}{4\pi^2|z^2-1|\left|D(\xi^{-1})\right|^4}
+o\left(|\xi|^{4n}\right).
\end{equation}

Therefore using \eqref{KNN} from the proof of Theorem \ref{OP} for a representation of $h_n^P$, and the expressions \eqref{Ksq} and \eqref{A0sq},  as $n\rightarrow \infty$ we have
\begin{align*}
h_n^P(z)&=\frac{1}{4\pi\left(\text{Im}(z)\right)^2}\left(1-\frac{\left|K_{n}(z,\bar{z})\right|^2}{\left(K_{n}(z,z)\right)^2}\right)\\
&= \frac{1}{4\pi\left(\text{Im}(z)\right)^2}
\left(1-\frac{\left(\frac{k_{n}}{k_{n+1}}\right)^2\frac{|\xi^{2n+1}|^2}{4\pi^2|z^2-1|\left|D(\xi^{-1})\right|^4}+o\left(|\xi|^{4n}\right)}
{\left(\frac{k_{n}}{k_{n+1}} \frac{\text{Im}(\xi)}{2\pi\text{Im}(z)}\right)^2\left| \frac{\xi^n}{D(\xi^{-1})} \right|^4+o(|\xi|^{4n})}\right)\\
&=\frac{1}{4\pi\left(\text{Im}(z)\right)^2}\left(1-\frac{\left(\text{Im}(z)\right)^2|\xi|^2+o(1)}{|z^2-1|(\text{Im}(\xi))^2+o(1)}\right)\\
&=\frac{1}{4\pi\left(\text{Im}(z)\right)^2}\left(1-\frac{\left(\text{Im}(z)\right)^2|z+\sqrt{z^2-1}|^2+o(1)}{|z^2-1|(\text{Im}(z+\sqrt{z^2-1} )^2+o(1)}\right),
\end{align*}
and hence completes the proof of Theorem \ref{GOP2}.
\end{proof}

\textbf{Acknowledgements.}  The author would like to thank his advisor Igor Pritsker for all his help with the project and for helping with financial support through his grant from the National Security Agency.  The author would also like to thank Jeanne LeCaine Agnew Endowed Fellowship and the Vaughn Foundation via Anthony Kable for financial support.


\begin{thebibliography}{00}

\bibitem{A}
R. J. Adler, The Geometry of Random Fields, Wiley, Chichester, 1981.

\bibitem{BY}
T. Bayraktar, Equidistribution of zeros of random holomorphic sections, Indiana Univ. Math. J., to appear.

\bibitem{BRS}
A. Bharucha-Reid and M. Sambandham, Random Polynomials, Academic Press, New York, 1986.

\bibitem{BP}
A. Bloch and G. P\'{o}lya, On the roots of a certain algebraic equation, Proc. Lond. Math. Soc. 33 (1932) 102--114.

\bibitem{BL1}
T. Bloom, Random polynomials and Green functions, Int. Math. Res. Not. 28 (2005) 1689--1708.

\bibitem{BL2}
T. Bloom, Random polynomials and (pluri) potential theory, Ann. Polon. Math. 91 (2007) 131--141.

\bibitem{BLL}
T. Bloom and N. Levenberg, Random polynomials and pluripotential-theoretic extremal functions, Potential Anal. 42 (2015) 311--334.

\bibitem{BLSH}
T. Bloom and B. Shiffman, Zeros of random polynomials on $\C^m$, Math. Res. Lett. 14 (2007) 469--479.

\bibitem{EK}
A. Edelman and E. Kostlan, How many zeros of a random polynomial are real?, Bull. Amer. Math. Soc. 32 (1995) 1-37.

\bibitem{D}
M. Das, Real zeros of a random zum of orthogonal polynomials, Proc. Amer. Math. Soc. 1 (1971) 147--153.

\bibitem{DB}
M. Das and S. Bhatt, Real roots of random harmonic equations, Indian J. Pure Appl. Math. 13 (1982) 411--420.

\bibitem{EO}
P. Erd\H{o}s and A. Offord, On the number of real roots of a random algebraic equation, Proc. Lond. Math. Soc. 6 (1956) 139--160.

\bibitem{F}
K. Farahmand, Random polynomials with complex coefficients, Stat. Prob. Lett. 27 (1996) 347--355.

\bibitem{FB}
K. Farahmand, Topics in Random Polynomials, Pitman Res. Notes in Math. Series 393, Addison Wesley Longman Limited, Edinburgh Gate, Harlow, England, 1998.

\bibitem{FG}
K. Farahmand and A. Grigorash, Complex zeros of trignometric polynomials with standard normal random coefficients, J.  Math. Anal. App. 262 (2001) 554--563.


\bibitem{KJ}
K. Farahmand and J. Jahangiri, Complex roots of a class of random algebraic polynomials, J. Math. Anal. App. 226 (1998) 220--228.

\bibitem{FL}
N. Feldheim, Zeros of Gaussian analytic functions with translation-invariant distribution, Israel J. Math. 195 (2012) 317--345.

\bibitem{HM}
J. Hammersley, The zeros of a random polynomial, Proc. of the Third Berk. Sym. on Math. Stat. and Prob. 1954-1955 vol. II, University of Cal. Press, Berkeley and Los Angeles (1956) 89--111.

\bibitem{ZGAF}
J. Hough, M. Krishnapur, Y. Peres, B. Vir$\acute{\text{a}}$g, Zeros of Gaussian Analytic Functions and Determinantal Point Processes, Univ. Lect. Ser. 51. American Mathematical Society, Providence, RI, 2009.

\bibitem{IZ}
I. Ibragimov and O. Zeitouni, On the roots of random polynomials, Trans. Amer. Math. Soc. 349 (1997) 2427--2441.

\bibitem{K1}
M.  Kac,  On  the  average  number  of  real  roots  of  a  random  algebraic  equation,  Bull.
Amer. Math. Soc. 49 (1943) 314--320.

\bibitem{K2}
M. Kac, On the average number of real roots of a random algebraic equation II, Proc.
Lond. Math. Soc. 50 (1948) 390--408.

\bibitem{AL}
A. Ledoan, Explicit formulas for the distribution of complex zeros of random polynomials, presentation at the $15^{\text{th}}$ Int. Conf.  Approx. Theory, San Antonio, TX, May 25, 2016.

\bibitem{LO}
J. Littlewood and A. Offord, On the number of real roots of a random algebraic equation, J. Lond. Math. Soc. 13 (1938) 288--295.

\bibitem{LPX}
D. Lubinsky, I. Pritsker, and X. Xie, Expected number of real zeros for random linear combinations of orthogonal polynomials, Proc.  Amer. Math. Soc. 144 (2016) 1631--1642.

\bibitem{LPX2}
D. Lubinsky, I. Pritsker, and X. Xie, Expected number of real zeros for random orthogonal polynomials, Math. Proc. Camb. Phil. Soc., to appear.

\bibitem{R}
S. Rice, Mathematical theory of random noise, Bell System Tech J. 25 (1945) 46--156.

\bibitem{RU}
W. Rudin, Real and Complex Analysis, third ed., McGraw-Hill, New York, NY, 1987.

\bibitem{SV}
L. A. Shepp and R. J. Vanderbei, The complex zeros of random polynomials, Trans. Amer. Math. Soc. 347 (1995) 4365--4384.

\bibitem{SHZ1}
B. Shiffman and S. Zelditch, Distribution of zeros of random and quantum chaotic sections of positive line bundles, Comm. Math. Phys. 200 no. 3 (1999)  661--683.

\bibitem{SHZ2}
B. Shiffman and S. Zelditch, Equilibrium distribution of zeros of random polynomials, Int. Math. Res. Not. 1 (2003) 25--29.

\bibitem{SHZ3}
B. Shiffman and S. Zelditch, Random complex fewnomials I, Notions of Positivity and the Geometry of Polynomials, Trends Math., Birkh\"{a}user/Springer Basel AG, Basel (2001), 375--400.

\bibitem{SZ}
G.~Szeg\H{o}, Orthogonal Polynomials, fourth ed., Amer. Math. Soc., Providence, RI,  1975.

\bibitem{TV}
T. Tao and V. Vu, Local universality of zeros of random polynomials, Int. Math. Res. Not. (2014) 1--84.

\bibitem{Ul}
D. Ullrich, Complex Made Simple, Amer. Math. Soc., Providence, RI, 2008.

\bibitem{CZRS}
R. J. Vanderbei, The complex zeros of random sums, arXiv: 1508.05162v1   Aug. 21, 2015.

\bibitem{WG}
Y. Wang, Bounds on the average number of real roots of a random algebraic equation (Chinese), Chinese Ann. Math. Ser. A 4 no. 5 (1983) 601--605.  An English summary appears in Chinese Ann. Math. Ser. B 4 no 4 (1983) 527.

\bibitem{WL}
J. Wilkins Jr., An asymptotic expansion for the expected number of real zeros of a random polynomial, Proc. Amer. Math. Soc. 103 (1988) 1249--1258.

\bibitem{AY}
A. Yeager, Zeros of random linear combinations of entire functions with complex Gaussian coefficients, arXiv: 1605.06836v1  May 22, 2016.

\end{thebibliography}
\end{document}